\documentclass[reqno,12pt]{amsart} 
\usepackage[a4paper,top=2.5cm,bottom=2.5cm,left=3cm,right=3cm]{geometry}
\usepackage{mathtools} 
\usepackage{amsfonts,amssymb,amsthm,stmaryrd}
\usepackage{mathrsfs} 
\usepackage{bbold} 
\usepackage{stackengine} 
\usepackage{bm} 
\usepackage{enumitem}
\usepackage{graphicx}
\usepackage{float} 
\usepackage{tikz} 
\usepackage{array}
\usepackage{tabularray}
\usepackage{multirow} 
\usepackage{xcolor}
\usepackage{hyperref}
\usepackage{cleveref} 
\usepackage{comment}
\usepackage[toc,page]{appendix}

\title{Large values of shifted mixed character sums}

\author[N. Tardy]{N\'eo Tardy}
\address{DER de math\'ematiques, ENS Paris-Saclay, 91190 Gif-sur-Yvette, France }
\email{neo.tardy@ens-paris-saclay.fr}

\date{}

\newtheorem{theorem}{Theorem}[section]
\newtheorem{proposition}[theorem]{Proposition}
\newtheorem{lemma}[theorem]{Lemma}

\theoremstyle{definition}

\theoremstyle{remark}

\numberwithin{equation}{section}

\begin{document}

\maketitle

\begin{abstract}
    We consider sums of the form
    $$F_\chi(\alpha,\beta;\theta) \vcentcolon= \sum_{\alpha p<n\le\beta p}\chi(n)e(n\theta),$$
    where $\chi$ is a non-principal Dirichlet character modulo a prime number $p$. We prove that
    $$
    \sqrt p \log \log p
    \ll \max_{0 \le \theta < 1}{\left|F_\chi(\alpha,\beta;\theta)\right|}
    \ll \sqrt{p}\log p,
    $$
    generalizing an old result of Montgomery as well as a recent result of Iggidr in two aspects: we allow general non-principal characters $\chi$, and we consider incomplete mixed character sums.
\end{abstract}

\section{Introduction}

The distribution of \textit{mixed character sums} of the form
$$ \sum_{n \le x} \chi(n)e(n\theta), $$
where $\chi$ is a Dirichlet character and $e(t) \vcentcolon= e^{2\pi i t}$, is a well-studied problem that has given rise to many recent developments. In this paper, we are interested in the case where one fixes $x$ and varies $\theta$ as in the case of the \textit{Fekete polynomials}
$$ F_p(\theta) \vcentcolon= \sum_{n=1}^{p-1} \left(\frac{n}{p}\right)e(n\theta),$$
where $p$ is prime. This particular case can be regarded as the simplest case to analyze, and its distribution has been widely studied (see e.g. \cite{borwein2002explicit,erdelyi2012upper,erdelyi2018improved,erdelyi2007large} for various bounds on the $L^q$-norms ($0<q<\infty$) and on the Mahler measure of $F_p$). When $q=2k$ is an even positive integer, the asymptotics of the $L^{2k}$-norms of the Fekete polynomials as $p\to+\infty$ were determined by Günther and Schmidt \cite{gunther2017lq} to be $\|F_p\|_{2k}\sim \phi_k\sqrt p$, where $\phi_k$ only depends on $k$. This was later generalized by Klurman, Lamzouri and Munsch \cite{klurman2023l_q} for all $q>0$. Their method is based on a randomization of $F_p(\theta)$, where they prove that, if $F_p(\theta)$ is viewed as a random continuous function defined on a suitable probability space, then it converges in distribution to a random function that can be computed explicitly. In particular, they were able to determine the asymptotic behaviour of the Mahler measure $M_0(F_p)=\lim\limits_{q\to0^+} \|F_p\|_{q}$ of the Fekete polynomials as $p\to+\infty$, proving that $M_0(F_p) \sim \phi_0\sqrt{p}$ for some absolute constant $\phi_0\approx0.749$.

Possible generalizations of the Fekete polynomials are \textit{Turyn polynomials}
$$ F_p^r(\theta) \vcentcolon= \sum_{n=1}^{p} \left(\frac{n+r}{p}\right)e(n\theta), $$
where $r$ is an integer, which we typically assume to be of the form $ \lfloor \alpha p \rfloor $ for some fixed $\alpha \in \mathbb{R}$, and \textit{shifted mixed character sums}
$$ F_\chi(\alpha,\beta;\theta) \vcentcolon= \sum_{\alpha p<n\le\beta p}\chi(n)e(n\theta). $$
Note that if $r = \lfloor \alpha p \rfloor$, we have $F_p^r(\theta) = F_{\left(\frac{\bullet}{p}\right)}(\alpha,\alpha+1;\theta)$. In \cite{gunther2017lq}, the authors extended their study of the $L^{2k}$-norms to the asymptotic behaviour of $\|F_p^r\|_{2k}$ when $\frac{r}{p}\to\alpha$. More generally, the distribution of $F_\chi(\alpha,\beta;\theta)$ when $\theta$ varies was understood in a recent paper by Bober, Klurman and Shala \cite{bober2025distribution} using a similar randomization trick as in \cite{klurman2023l_q}.

A natural question that arises in connection with the distribution concerns large values, which provide information about the tail of the distribution of the values of $F_\chi(\alpha,\beta;\theta)$ as $\theta$ varies. The first result in this direction is due to Montgomery \cite{montgomery1980exponential} and was proved for Fekete polynomials. He showed that for all large enough $p$, we have
\begin{equation}\label{montgomerybounds}
    \sqrt{p}\log \log p \ll \max\limits_{0 \le \theta < 1} |F_p(\theta)| \ll \sqrt{p}\log p,
\end{equation}
and conjectured that
$$ \max\limits_{0 \le \theta < 1} |F_p(\theta)| \asymp \sqrt{p}\log \log p. $$
Later, Conrey, Granville, Poonen and Soundararajan \cite{conrey2000zeros} gave an alternative proof of \eqref{montgomerybounds}. Very recently, Iggidr \cite[Corollary 1.4]{iggidr2026distribution} generalized the lower bound of \eqref{montgomerybounds} to all non-principal Dirichlet characters $\chi$ mod $p$ by proving that
$$ \max\limits_{0 \le \theta < 1} |F_\chi(0,1;\theta)| \gg \sqrt{p}\log \log p, $$
and claimed that his method can provide the same bound for Turyn polynomials $F_\chi(\alpha,\alpha+1;\theta)$. In fact, the result in \cite{iggidr2026distribution} is more general, and describes the exact distribution of the large values of $F_\chi(0,1;\theta)$.

Since a naive heuristic suggests that Dirichlet characters behave like random variables uniformly distributed on their support, it is interesting to compare \eqref{montgomerybounds} to known bounds for the random analogue
$$ Q_N(\theta) \vcentcolon= \sum_{n=1}^{N} X_n e(n\theta), $$
where the $X_n$ are independent Rademacher random variables. In this setting, it was proved by Salem and Zygmund \cite{salem1954some} that, almost surely, for $N$ large enough,
$$ \max\limits_{0 \le \theta < 1} |Q_N(\theta)| \asymp \sqrt{N\log N}, $$
later improved by Hal\'asz \cite{halasz1973result} to $\sim \sqrt{N\log N}$. A more precise model would take into account the multiplicativity of characters, and we set
$$ P_N(\theta) \vcentcolon= \sum_{n=1}^{N} f(n) e(n\theta), $$
where $f$ is a random multiplicative function. If $f$ is a Rademacher or a Steinhaus random multiplicative function, it was proved by Benatar, Nishry and Rodgers \cite{benatar2022moments} that, almost surely, for $N$ large enough,
$$
\sqrt{N}\left(\frac{\log N}{\log \log N}\right)^{1/6}
\ll \max\limits_{0 \le \theta < 1} |P_N(\theta)|
\ll \sqrt{N} \exp \left(3\sqrt{\log N\log \log N}\right).
$$
This was recently improved by Hardy \cite{hardy2024bounds}, who proved that, almost surely, for $N$ large enough,
$$ \max\limits_{0 \le \theta < 1} |P_N(\theta)| \gg \sqrt{N \log N}, $$
and conjectured that this lower bound is of the correct order of magnitude.
Surprisingly, these bounds are inconsistent with the conjectured size of $\max\limits_{0 \le \theta < 1} |F_p(\theta)|$, which suggests that Dirichlet characters exhibit more cancellation than a typical random multiplicative function. Indeed, $P_N(\theta)$ is expected to behave like a complex Gaussian random variable when $\theta$ varies (see \cite[Theorem 1.2]{benatar2022moments}). However, the limiting distribution of $F_p(\theta)$ (or more generally of $F_\chi(\alpha,\beta;\theta)$) discovered in \cite{bober2025distribution,klurman2023l_q} is not Gaussian. This may be explained by the fact that Dirichlet characters exhibit analytic structures that are absent from random multiplicative functions (see Proposition~\ref{prop:BKS}).

In this paper, we generalize Montgomery's bounds \eqref{montgomerybounds} to all shifted mixed character sums.

\begin{theorem}\label{thm:main}
    Let $p$ be a sufficiently large prime, $\chi$ a non-principal Dirichlet character mod $p$, and $\beta>\alpha\geq0$. Then, 
    $$
    \sqrt p \log \log p
    \ll_{\alpha,\beta} \max_{0 \le \theta < 1}{\left|F_\chi(\alpha,\beta;\theta)\right|}
    \ll_{\alpha,\beta} \sqrt{p}\log p.
    $$
\end{theorem}

Our first step in the proof of Theorem \ref{thm:main} relies on the following randomization trick introduced by Klurman, Lamzouri and Munsch \cite{klurman2023l_q} and then used by Bober, Klurman and Shala \cite{bober2025distribution} and Iggidr \cite{iggidr2026distribution} to study mixed character sums. The key idea is to introduce randomness by looking at $F_\chi(\alpha,\beta;\theta)$ as the continuous random process
$$
k \mapsto F_\chi\left(\alpha,\beta;\frac{k+t}{p}\right),
$$
where $0 \le t < 1$. Using probabilistic methods, they study its limiting distribution as $p\to+\infty$, which takes the form of an infinite sum of independent random variables. In the present case, we do not need to invoke the probabilistic interpretation, but we use the aforementioned idea by writing
\begin{equation}\label{maxmax}
    \max_{0 \le \theta < 1}{\left|F_\chi(\alpha,\beta;\theta)\right|}
    = \max_{0 \le k \le p-1} \max_{0 \le t < 1} \left|F_\chi\left(\alpha,\beta;\frac{k+t}{p}\right)\right|.
\end{equation}
This allows us to use the following proposition due to Bober, Klurman and Shala \cite[Proposition 3]{bober2025distribution}, which provides a formula for handling incomplete sums.

\begin{proposition}\label{prop:BKS}
    Let $p$ be a prime, $\chi$ a non-principal Dirichlet character mod $p$, $\beta>\alpha\ge0$, $0 \le k \le p-1$, $t \in (0,1)$ and $K \ge 1$. Then,
    \begin{equation}\label{BKShformula}
        \begin{multlined}
            F_\chi\left(\alpha,\beta;\frac{k+t}{p}\right)
            = \frac{\tau(\chi)e(\alpha k)}{2\pi i} \sum_{\left|l\right|<K}{\frac{e\left(\beta\left(l+t\right)\right)-e\left(\alpha\left(l+t\right)\right)}{l+t}\overline{\chi}\left(k-l\right)} \\
            + O\left(\frac{p\log{p}}{K}+1\right),
        \end{multlined}
    \end{equation}
    where $\tau(\chi)$ is the Gauss sum associated with $\chi$.
\end{proposition}

Note that, since a non-principal character $\chi$ modulo a prime number $p$ is a primitive character, we have $|\tau(\chi)|=\sqrt{p}$. Thus, the upper bound in Theorem \ref{thm:main} follows immediately from Proposition \ref{prop:BKS} with $K=p$.

\subsection*{Outline of the proof}
Our proof uses similar ideas to those of Montgomery \cite{montgomery1980exponential} and Conrey, Granville, Poonen and Soundararajan \cite{conrey2000zeros}, incorporating Proposition~\ref{prop:BKS} as a new ingredient. We observe in Montgomery's proof that the value of $t \in (0,1)$ does not affect the size of the maximum in \eqref{maxmax}. Thus, for the lower bound, we view $t$ as being fixed. Then, we shall prove that there exists an integer $k \in [0,p-1]$ such that $F_\chi\left(\alpha,\beta;\frac{k+t}{p}\right) \gg_{\alpha,\beta} \sqrt{p}\log \log p$. To do so, we will prove that there is a set $\mathcal{S}$ of values of $k$ for which we can prescribe the values of $\chi(k-l)$ for all $|l| \le K_0$ for some $K_0 \ge 1$. With a suitable choice of the values, we will be able to prove that the main term of \eqref{BKShformula} has size $\gg \sqrt{p}\log K_0$. Thus, in the first instance, we shall view $K_0$ as a bounded power of $\log p$. We will construct the set $\mathcal{S}$ in Lemma \ref{lem:key}, and the key ingredient for that is the following proposition (see e.g. \cite[Theorem 11.23]{iwaniec2021analytic}).

\begin{proposition}[Weil's bound for character sums]\label{prop:Weil}
    Let $p$ be a prime and $\chi$ a non-principal Dirichlet character mod $p$ of order $d$. Suppose $P\in\mathbb{F}_p\left[X\right]$ has $m$ distinct roots and $P$ is not a $d$-th power. Then,
    $$ \left|\sum_{k=0}^{p-1}\chi\left(P(k)\right)\right|\le\left(m-1\right)\sqrt p. $$
\end{proposition}

With the above choice of $K=K_0$, the error term in Proposition \ref{prop:BKS} would be too large compared to the main term. Nevertheless, we will see that this can be made smaller after averaging over $k \in \mathcal{S}$.

\subsection*{Notation}

For $x \in \mathbb{R}$, we denote by $\{x\}=x-\lfloor x \rfloor$ its fractional part and $\|x\|$ its minimal distance to an integer. We also use the standard notation $e(x) \vcentcolon= e^{2\pi i x}$. For $d \ge 1$, we set $\mu_d \vcentcolon= \{ z \in \mathbb{C} : z^d=1 \}$.

\subsection*{Acknowledgments}

I would like to thank Oleksiy Klurman for suggesting this problem to me and for his guidance throughout the project. I am also grateful to Besfort Shala for numerous helpful discussions, and to Seth Hardy and Marc Munsch for their comments on an earlier draft of the paper. Finally, I thank the University of Bristol for providing excellent working conditions.

\section{Proof of Theorem \ref{thm:main}}

We have seen that the upper bound in Theorem \ref{thm:main} is a direct consequence of Proposition \ref{prop:BKS}. Thus, we now turn to the proof of the lower bound.

Let $t \in [0.1,0.9]$ be fixed. Let us take a large prime $p$ and $\beta>\alpha\ge0$. Take $K_0 \vcentcolon= \left\lfloor\frac{\log{p}}{\left(\log{\log{p}}\right)^2}\right\rfloor$ and define for all integers $k \in [0,p-1]$ and Dirichlet characters $\chi$ mod $p$,
$$
F_{\chi,k}(\alpha,\beta)
\vcentcolon= F_\chi\left(\alpha,\beta;\frac{k+t}{p}\right)
$$
and
$$
{\widetilde{F}}_{\chi,k}(\alpha,\beta)
\vcentcolon=\frac{\tau(\chi)e(\alpha k)}{2\pi i} \sum_{\left|l\right|\le K_0}{\frac{e\left(\beta\left(l+t\right)\right)-e\left(\alpha\left(l+t\right)\right)}{l+t}\overline{\chi}\left(k-l\right)}.
$$

In the next lemma, we prove that there exists a choice of $k$ for which the values of $\chi(k-l)$ can be prescribed for all $|l| \le K_0$, and for which $F_{\chi,k}(\alpha,\beta)$ is well approximated by ${\widetilde{F}}_{\chi,k}(\alpha,\beta)$.

\begin{lemma}\label{lem:key}
    Let $\chi$ be a non-principal Dirichlet character mod $p$ of order $d$ and consider a collection $\left(\xi_l\right)_{\left|l\right|\le K_0}$ of elements of $\mu_d$. Let
    $$ \mathcal{S} \vcentcolon= \bigcap_{\left|l\right|\le K_0} \bigcup_{|a| \le \frac{d}{20}} \left\{ 0 \le k \le p-1 : \chi\left(k-l\right)=\xi_le\left(\frac{a}{d}\right) \right\}. $$
    Then, there exists $k \in \mathcal{S}$ such that
    \begin{equation}\label{lem:key-0}
        F_{\chi,k}(\alpha,\beta)={\widetilde{F}}_{\chi,k}(\alpha,\beta)+O(\sqrt{p}).
    \end{equation}
\end{lemma}
\begin{proof}
    Define
    $$ W(k) \vcentcolon= \prod_{\left|l\right|\le K_0}\sum_{\left|a\right|\le\frac{d}{20}}\sum_{j=0}^{d-1}\left(\overline{\chi}\left(k-l\right)\xi_le\left(\frac{a}{d}\right)\right)^j. $$
    Since the union in the definition of $\mathcal{S}$ is disjoint, by orthogonality in $\mu_d$, we have
    \begin{equation}\label{lem:key-1}
        W(k) = d^{2K_0+1}\mathbb{1}_{\mathcal{S}}(k) + O\left(d^{2K_0+1} \mathbb{1}_{[0,K_0] \cup [p-K_0,p-1]}(k)\right).
    \end{equation}
    Next, expand the product that defines $W(k)$ to obtain
    $$
    W(k) = \sum_{P\in\mathcal{P}}  \overline{\chi}\left(P(k)\right) \prod_{\left|l\right|\le K_0}{\xi_l^{j_{l,P}}\sum_{\left|a\right|\le\frac{d}{20}} e\left(\frac{aj_{l,P}}{d}\right)},
    $$
    where $\mathcal{P}$ is the set of all polynomials of the form $P(x)=\prod_{|l|\le K_0} (x-l)^{j_{l,P}}$ with $0 \le j_{l,P} \le d-1$ for all $|l| \le K_0$. It follows from Weil’s bound (Proposition~\ref{prop:Weil}) that
    \begin{equation*}
        \begin{aligned}
            \sum_{k=0}^{p-1}W(k)
            &= p\left(2\left\lfloor\frac{d}{20}\right\rfloor+1\right)^{2K_0+1}+O\left(\sum_{P\in\mathcal{P}\backslash\{1\}}{\left|\sum_{k=0}^{p-1}{\overline{\chi}\left(P(k)\right)}\right|\prod_{\left|l\right|\le K_0}\left|\sum_{\left|a\right|\le\frac{d}{20}} e\left(\frac{aj_{l,P}}{d}\right)\right|}\right) \\
            &= p\left(2\left\lfloor\frac{d}{20}\right\rfloor+1\right)^{2K_0+1}+O\left(K_0\sqrt p\sum_{P\in\mathcal{P}\backslash\{1\}}\prod_{\left|l\right|\le K_0}\left|\sum_{\left|a\right|\le\frac{d}{20}} e\left(\frac{aj_{l,P}}{d}\right)\right|\right).
        \end{aligned}
    \end{equation*}
    By the inequality $\sum_{|a|\le x} e(a\theta) \ll \frac{1}{\|\theta\|}$, we have
    \begin{equation*}
        \begin{aligned}
            \sum_{P\in\mathcal{P}\backslash\{1\}}\prod_{\left|l\right|\le K_0}\left|\sum_{\left|a\right|\le\frac{d}{20}} e\left(\frac{aj_{l,P}}{d}\right)\right|
            &\le \sum_{\substack{(j_l)_{|l| \le K_0} \\ \forall l, 0 \le j_l \le d-1}} \prod_{\left|l\right|\le K_0}\left|\sum_{\left|a\right|\le\frac{d}{20}} e\left(\frac{aj_l}{d}\right)\right| \\
            &= \prod_{\left|l\right|\le K_0}\sum_{j=0}^{d-1}\left|\sum_{\left|a\right|\le\frac{d}{20}} e\left(\frac{aj}{d}\right)\right| \\
            &\ll \left(C_1\left(d+\sum_{1\le j\le\frac{d}{2}}\frac{d}{j}+\sum_{\frac{d}{2}<j\le d-1}\frac{d}{d-j}\right)\right)^{2K_0+1} \\
            &\ll \left(C_2d\log{d}\right)^{2K_0+1}
        \end{aligned}
    \end{equation*}
    for some constants $C_1,C_2>0$. Thus,
    \begin{align}\label{lem:key-2}
        \sum_{k=0}^{p-1}W(k)
        &= p\left(2\left\lfloor\frac{d}{20}\right\rfloor+1\right)^{2K_0+1}+O\left(K_0\sqrt p\left(C_2 d\log{d}\right)^{2K_0+1}\right) \nonumber \\
        &\asymp p\left(2\left\lfloor\frac{d}{20}\right\rfloor+1\right)^{2K_0+1},
    \end{align}
    since $d \le p$ and by our choice of $K_0$. In particular, it follows from \eqref{lem:key-1} that
    \begin{align}\label{lem:key-2.5}
        |\mathcal{S}|
        &= \frac{1}{d^{2K_0+1}} \sum_{K_0 < k < p-K_0} W(k)
        = \frac{1}{d^{2K_0+1}} \sum_{k=0}^{p-1} W(k) + O(K_0) \nonumber \\
        &\asymp \frac{p}{d^{2K_0+1}}\left(2\left\lfloor\frac{d}{20}\right\rfloor+1\right)^{2K_0+1}.
    \end{align}
    
    Now, we shall prove that
    \begin{equation}\label{lem:key-2.1}
        \frac{1}{\left|\mathcal{S}\right|}\sum_{k\in\mathcal{S}}\left|F_{\chi,k}(\alpha,\beta)-{\widetilde{F}}_{\chi,k}(\alpha,\beta)\right| \ll \sqrt{p}.
    \end{equation}
    Let $K > K_0$ be an integer, by the triangle inequality, we have for all integers $k \in [0,p-1]$,
    \begin{equation*}
        \begin{multlined}
            \left|F_{\chi,k}(\alpha,\beta)-{\widetilde{F}}_{\chi,k}(\alpha,\beta)\right|
            \le \frac{\sqrt p}{2\pi}\left|\sum_{K_0<\left|l\right|\le K}{\frac{e\left(\beta\left(l+t\right)\right)-e\left(\alpha\left(l+t\right)\right)}{l+t}\overline{\chi}\left(k-l\right)}\right| \\
            + \left|F_{\chi,k}(\alpha,\beta)-\frac{\tau(\chi)e(\alpha k)}{2\pi i}\sum_{\left|l\right|\le K}{\frac{e\left(\beta\left(l+t\right)\right)-e\left(\alpha\left(l+t\right)\right)}{l+t}\overline{\chi}\left(k-l\right)}\right|.
        \end{multlined}
    \end{equation*}
    Take $K=p-K_0-1$. Then, the second sum is $\ll\log{p}$ by Proposition \ref{prop:BKS} and the first sum is
    $$ \ll\sqrt p\left(\left|\sum_{K_0<\left|l\right| < p-K_0}{\frac{e(\beta l)}{l+t}\overline{\chi}\left(k-l\right)}\right|+\left|\sum_{K_0<\left|l\right| < p-K_0}{\frac{e(\alpha l)}{l+t}\overline{\chi}\left(k-l\right)}\right|\right), $$
    hence it suffices to prove that for $\epsilon\in\left\{-1,+1\right\}$ and $x\in\mathbb{R}$, we have
    \begin{equation}\label{lem:key-3}
        \frac{1}{\left|\mathcal{S}\right|}\sum_{k\in\mathcal{S}}\left|\sum_{K_0<\epsilon l<p-K_0}{\frac{e(lx)}{l}\overline{\chi}(k-l)}\right|\ll1.
    \end{equation}
    Expressing the indicator function of $\mathcal{S}$ as in \eqref{lem:key-1} and using \eqref{lem:key-2.5}, we have
    \begin{equation}\label{lem:key-4}
        \begin{aligned}
            &\frac{1}{\left|\mathcal{S}\right|}\sum_{k\in\mathcal{S}}\left|\sum_{K_0<\epsilon l<p-K_0}{\frac{e(lx)}{l}\overline{\chi}(k-l)}\right|^2 \\
            &\ll \sum_{K_0<\epsilon l<p-K_0}\frac{1}{l^2} \\
            &\quad +\frac{1}{p\left(2\left\lfloor\frac{d}{20}\right\rfloor+1\right)^{2K_0+1}}\sum_{K_0<\epsilon l_1\neq\epsilon l_2<p-K_0}{\frac{1}{l_1 l_2}\left|\sum_{K_0<k<p-K_0}W(k)\overline{\chi}(k-l_1)\chi(k-l_2)\right|}.
        \end{aligned}
    \end{equation}
    Similarly to \eqref{lem:key-2}, when $l_1\neq l_2$ with $K_0<\left|l_1\right|,\left|l_2\right|<p-K_0$, we have
    \begin{equation}\label{lem:key-5}
        \begin{multlined}
            \left|\sum_{K_0<k<p-K_0}W(k)\overline{\chi}(k-l_1)\chi(k-l_2)\right| \\
            = \left|\sum_{k=0}^{p-1}W(k)\overline{\chi}(k-l_1)\overline{\chi}^{d-1}(k-l_2)\right| + O\left(K_0 d^{2K_0+1}\right) 
            \ll K_0\sqrt p\left(Cd\log{d}\right)^{2K_0+1}
        \end{multlined}
    \end{equation}
    for some constant $C>0$. Indeed, one may replace the set $\mathcal{P}$ by the set $\mathcal{P}'$ of all polynomials of the form $P(x)=(x-l_1)(x-l_2)^{d-1}\prod_{|l|\le K_0} (x-l)^{j_{l,P}}$ with $0 \le j_{l,P} \le d-1$ for all $|l| \le K_0$. By assumption on $l_1$ and $l_2$, elements of $\mathcal{P}' \backslash \{1\}$ are not $d$-th powers in $\mathbb{F}_p[X]$, hence Weil's bound may be applied. Inserting \eqref{lem:key-5} into \eqref{lem:key-4} and applying the Cauchy-Schwarz inequality, we obtain \eqref{lem:key-3}, and consequently \eqref{lem:key-2.1}. It follows that there exists $k \in \mathcal{S}$ that satisfies \eqref{lem:key-0}.
\end{proof}

We now state another technical lemma before proving Theorem~\ref{thm:main}.

\begin{lemma}\label{lem:sumcos}
    Fix $u_1,u_2,c_1,c_2\in\mathbb{R}$ with $c_1,c_2\notin\frac{1}{2}\mathbb{Z}$. Then, for all $K\geq2$, we have
    $$ \sum_{1\le l\le K}\frac{\left|\sin{\left(\pi\left(u_1l+c_1\right)\right)}\sin{\left(\pi\left(u_2l+c_2\right)\right)}\right|}{l}\gg \log{K}, $$
    where the implicit constant only depends on the fractional parts of $u_1$, $u_2$, $c_1$ and $c_2$.
\end{lemma}
\begin{proof}
   Without loss of generality, assume that $0 \le u_1,u_2,c_1,c_2 < 1$. Let $x \ge 1$ be large enough depending on $u_1$ and $c_1$. Then, one can observe that the number of integers $l \in [x,2x]$ such that $\|u_1 l+c_1\| \ge \frac{\|2c_1\|}{20}$ is $\ge 0.6x$, since the sequence $(u_1 l+c_1)_l$ is equidistributed on the closure of its image mod $1$. Since a similar statement holds with $u_2l+c_2$, we deduce that for $x$ large enough depending on $u_1$, $u_2$, $c_1$ and $c_2$, we have $\left|\sin{\left(\pi\left(u_1l+c_1\right)\right)}\right|\gg_{c_1}1$ and $\left|\sin{\left(\pi\left(u_2l+c_2\right)\right)}\right|\gg_{c_2}1$ for $\ge 0.2x$ values of $l\in\left[x,2x\right]$, and the lemma follows by dyadic decomposition.
\end{proof}

\begin{proof}[Proof of the lower bound in Theorem~\ref{thm:main}]
    It suffices to prove that there exists an integer $k \in [0,p-1]$ such that \eqref{lem:key-0} holds and
    \begin{equation*}\label{thm:main-1}
        \left|{\widetilde{F}}_{\chi,k}(\alpha,\beta)\right| \gg_{\alpha,\beta} \sqrt{p}\log{K_0},
    \end{equation*}
    for some choice of $t \in [0.1,0.9]$ depending at most on $\alpha$ and $\beta$.

    First, assume that $d\geq3$. Then, for all $\left|l\right|\le K_0$, choose $\xi_l\in\mu_d$ such that
    $$ -\frac{\pi}{d}\le\arg{\left(\frac{\left(e\left(\beta\left(l+t\right)\right)-e\left(\alpha\left(l+t\right)\right)\right)\overline{\xi_l}}{l+t}\right)}\le\frac{\pi}{d}, $$
    with the convention that $\arg{0}=0$. Then, let $k$ be given by Lemma~\ref{lem:key}, since $\left|\tau(\chi)\right|=\sqrt p$, we have
    \begin{align}\label{thm:main-2}
        \left|{\widetilde{F}}_{\chi,k}(\alpha,\beta)\right|
        &\geq \frac{\sqrt{p}}{2\pi}\sum_{|l| \le K_0}\frac{\left|e\left(\beta\left(l+t\right)\right)-e\left(\alpha\left(l+t\right)\right)\right|\cos{\left(\frac{\pi}{3}+\frac{2\pi}{20}\right)}}{\left|l+t\right|} \nonumber \\
        &\gg \sqrt{p}\sum_{1\le l\le K_0}\frac{\left|\cos{\left(2\pi\left(\beta\left(l+t\right)\right)\right)}-\cos{\left(2\pi\left(\alpha\left(l+t\right)\right)\right)}\right|}{l}.
    \end{align}

        If $d=2$, then $\chi$ is a real character and \eqref{thm:main-2} holds for some $k$ provided by Lemma~\ref{lem:key}, after taking $\xi_l=\operatorname{sgn}{\left(\frac{\cos{\left(2\pi\left(\beta\left(l+t\right)\right)\right)}-\cos{\left(2\pi\left(\alpha\left(l+t\right)\right)\right)}}{l+t}\right)}$, with the convention that $\operatorname{sgn}{(0)}=1$.

        Now, \eqref{thm:main-2} yields
        $$ \left|{\widetilde{F}}_{\chi,k}(\alpha,\beta)\right|\gg_t\sum_{1\le l\le K_0}\frac{\left|\sin{\left(\pi\left(\alpha+\beta\right)\left(l+t\right)\right)}\sin{\left(\pi\left(\beta-\alpha\right)\left(l+t\right)\right)}\right|}{l}, $$
        hence it suffices to choose $t\in[0.1,0.9]$ such that $\left(\alpha+\beta\right)t, \left(\beta-\alpha\right)t\notin\frac{1}{2}\mathbb{Z}$ to conclude with Lemma~\ref{lem:sumcos}.
\end{proof}

\bibliographystyle{plain}
\bibliography{bibliography}

\end{document}